\documentclass{amsart}
\usepackage[T1]{fontenc}   
\usepackage{tikz}
\usepackage{amsmath}
\usepackage{amsthm}
\usepackage{amssymb}
\usepackage{mathabx}
\usepackage{mathrsfs} 
\usepackage[all]{xy} 
\usepackage{graphicx}

\newtheorem{theorem}{Theorem}[section]
\newtheorem{lemma}[theorem]{Lemma}
\newtheorem{proposition}[theorem]{Proposition}
\theoremstyle{definition}

\newtheorem{remark}[theorem]{Remark}
\numberwithin{equation}{section}

\DeclareMathOperator{\Lk}{Lk} 
\DeclareMathOperator{\St}{St} 
\DeclareMathOperator{\CAT}{CAT} 
\DeclareMathOperator{\Aut}{Aut} 
\DeclareMathOperator{\Isom}{Isom} 
\DeclareMathOperator{\Fix}{Fix} 

\begin{document}

\title[$\CAT(0)$ Extensions of RACGs]{$\CAT(0)$ Extensions of Right-angled Coxeter Groups}

\author[Cunningham]{Charles Cunningham}
\address{Department of Mathematics; Bowdoin College;
Brunswick, Maine 04011}
\email{charles.cunningham@tufts.edu}


\author[Eisenberg]{Andy Eisenberg}
\address{Department of Mathematics; Oklahoma State University;
Stillwater, OK 74078}
\email{andy.eisenberg@okstate.com}

\author[Piggott]{Adam Piggott}
\address{Department of Mathematics; Bucknell University;
Lewisburg, Pennsylvania 17837}
\email{adam.piggott@bucknell.edu}
\thanks{This work was partially supported by a grant from the Simons Foundation (\#317466 to Adam Piggott)}

\author[Ruane]{Kim Ruane}
\address{Department of Mathematics; Tufts University;
Medford, Massachusetts 02155}
\email{kim.ruane@tufts.edu}


\thanks {We thank the anonymous referee for their constructive suggestions which helped improve the clarity of our exposition.}

\begin{abstract}
We show that any split extension of a right-angled Coxeter group $W_{\Gamma}$ by a generating automorphism of finite order acts faithfully and geometrically on a $\CAT(0)$ metric space.
\end{abstract}

\maketitle

\section{Introduction}\label{sec:Introduction}

An isometric group action is \emph{faithful} if its  kernel is trivial, and it is \emph{geometric} if it is cocompact and properly discontinuous.  A finitely generated group $G$ is a \emph{CAT(0) group} if there exists a CAT(0) metric space $X$ equipped with a faithful geometric $G$-action.  The CAT(0) property is not an invariant of the quasi-isometry class of a group (see, for example, \cite{BridsonAlonso, KapovichLeeb} and \cite[p.~258]{MartinsBook}).  Whether or not it is an invariant of the abstract commensurability class of a group is
as yet unknown.  Attention was brought to this matter in \cite{PRW}.
In this article we illustrate that answering this question for any family of CAT(0) groups may require a variety of techniques.

It is well-known that an arbitrary right-angled Coxeter group $W$ is a CAT(0) group because it acts faithfully and geometrically on a CAT(0) cube complex $X$.  It is also well-known that the automorphism group $\Aut(W)$ is generated by three types of finite-order automorphisms.  As a natural source of examples we consider split extensions of right-angled Coxeter groups by finite cyclic groups, where in each case the cyclic group acts on $W$ as the group generated by one of these various generating automorphisms.  Our theorem is the following:

\begin{theorem}\label{thm:MainTheorem}
Suppose $W$ is a right-angled Coxeter group and $\phi \in \Aut(W)$ is either an automorphism induced by a graph automorphism, a partial conjugation, or a transvection.  Let $m$ denote the order of $\phi$.  Then the group $G = W \rtimes_\phi \mathbb{Z} / m\mathbb{Z}$ is a CAT(0) group.
\end{theorem}


What is most interesting is that $G$ is a CAT(0) group for different reasons in each of the three cases.  When  $\phi$ is an automorphism induced by a graph automorphism, the left-multiplication action $W \circlearrowright X$ extends to an action $G \circlearrowright X$; when $\phi$ is a partial conjugation, $G$ is itself a right-angled Coxeter group; when $\phi$ is a transvection, $G$ is not a right-angled Coxeter group and the action $W\circlearrowright X$ cannot extend to all of $G$, but we can explicitly construct a new CAT(0) space $Y$ and describe a faithful geometric action $G \circlearrowright Y$.

After necessary background material is described in Section \ref{sec:RACG}, the three cases of the theorem are treated, in turn, in Sections \ref{sec:graphautomorphism}, \ref{sec:partialconjugation} and \ref{sec:transvection}.

We also note that, in each case of the theorem, we take an extension $W_{\Gamma} \rtimes H$ where $H \leq \Aut(W_{\Gamma})$ is finite. In \cite{OtherPreprint}, we give an example in which $H$ is infinite and $W_{\Gamma} \rtimes H$ is not a right-angled Coxeter group.  We currently do not know whether such extensions with infinite $H$ are CAT(0) or not.  Since this question does not address the abstract commensurability of the CAT(0) property, we will not address it further in this paper.

\section{Right-angled Coxeter groups and their automorphisms}
\label{sec:RACG}

In this section we briefly recall a very small part of the rich combinatorial and geometric theory of right-angled Coxeter groups.  The interested reader may consult \cite{Davis} for a thorough account of the more general subject of Coxeter groups from the geometric group theory point of view.

Fix an arbitrary finite simple graph $\Gamma$ with vertex set $S$ and edge set $E$.  The \emph{right-angled Coxeter group defined by $\Gamma$} is the group $W = W_{\Gamma}$ generated by $S$, with relations declaring that the generators all have order 2, and adjacent vertices commute with each other. The pair $(W, S)$ is called a right-angled Coxeter system.  As described in \cite[Proposition 7.3.4, p.~130]{Davis}, we construct a cube complex $X = X(W, S)$ inductively as follows:
\begin{itemize}
\item The set of vertices is indexed by $W$, say $X^0 = \{v_w \mid w \in W\}$.
\item To complete the construction of the one-skeleton $X^1$ we add edges of unit length so that vertices $v_u, v_w$ are adjacent if and only if $u^{-1}w \in S$.
\item For each $k \geq 2$, we construct the $k$-skeleton by gluing in Euclidean unit cubes of dimension $k$ whenever $X^{k-1}$ contains the $(k-1)$-skeleton of such a cube.
\end{itemize}

\begin{remark}
We note the following about this construction:
\begin{itemize}
\item The dimension of $X$ equals the number of vertices in the largest clique in $\Gamma$.
\item The barycentric subdivision of $X$ is the well-known Davis complex $\Sigma = \Sigma(W, S)$.  By a result of Gromov, $\Sigma$, and hence also $X$, is a CAT(0) metric space (see \cite[Theorem 12.3.3, p.\ 235]{Davis} for a generalization due to Moussong).
\end{itemize}
\end{remark}

By construction, the geometry of $X$ is determined entirely by its 1-skeleton $X^1$.  It follows that a permutation $\sigma$ of the vertex set $X^0$ determines an isometry of $X$ if it respects the adjacency relation.  In particular, for all $w \in W$ the map $v_u \mapsto v_{wu}$ extends to an isometry $\Phi_w \in \Isom(X)$. The map $w \mapsto \Phi_w$  is a faithful geometric action $W \circlearrowright X$ known as the \emph{left-multiplication action}.

From the graph $\Gamma$ we may infer the existence of certain finite-order automorphisms of $W$.  For each vertex $a \in S$, we write $\Lk(a)$ for the set of vertices adjacent to $a$, and $\St(a)$ for $\Lk(a) \cup \{a\}$.
\begin{itemize}
\item Each graph automorphism $f \in \Aut(\Gamma)$ restricts to a permutation of $S$ which determimes an automorphism $\phi_f \in \Aut(W)$.
\item For each union of non-empty connected components $D$ of $\Gamma \setminus \St(a)$, the map
\[
s \mapsto \begin{cases} asa & s \in D, \\ s & s \in S \setminus D, \end{cases}
\]
determines an automorphism of $W$ called the \emph{partial conjugation with acting letter $a$ and domain $D$}.
\item If $a, d \in S$ are such that $\St(d) \subseteq \St(a)$, then the rule
\[
s \mapsto s \text{ for all } s \in S \setminus \{d\}, \;\;\; d \mapsto da,
\]
determines an automorphism of $W$ called the \emph{transvection with acting letter $a$ and domain $d$}.
\end{itemize}
Together, the automorphisms induced by graph automorphisms, the partial conjugations and the transvections comprise a generating set for $\Aut(W)$ \cite{LaurenceThesis}.
We note that partial conjugations and transvections are involutions, and graph automorphisms have finite order.

In what follows, $\phi \in \Aut(W)$ shall always denote a non-trivial automorphism of finite order $m$, and $G$ shall denote the semi-direct product $G = W \rtimes_\phi \mathbb{Z} / m\mathbb{Z}$.  So $G$ is presented by:
\begin{align*}
P_1 = \langle S \cup \{z\} &\mid s^2 = 1 \text{ for all } s \in S, [s,t] = 1 \text{ for all } \{s, t\} \in E, \\
&\quad z^m = 1, zsz^{-1} = \phi(s) \text{ for all } s \in S \rangle.
\end{align*}

\section{When $\phi$ is induced by a graph automorphism}\label{sec:graphautomorphism}

Suppose $\phi$ is induced by a graph automorphism $f \in \Aut(\Gamma)$. Then the map $v_w \mapsto v_{\phi(w)}$ preserves the adjacency relation in $X^1$, and hence determines an isometry $\Phi \in \Isom(X)$.  By simple computation the reader may confirm that the relations in the presentation $P_1$ are satisfied when each $s \in S$ is replaced by $\Phi_s$, and $z$ is replaced by $\Phi$.  Hence the rule
\[
s \mapsto \Phi_s \text{ for all } s \in S, z \mapsto \Phi,
\]
determines an action $G \circlearrowright X$.  We leave the reader to confirm that the action is faithful and geometric, and hence Theorem \ref{thm:MainTheorem} holds in the first of the three cases.

In fact, a stronger result holds for similar reasons.

\begin{lemma}\label{lem:graphautomorphismcase}
If $\mathcal{H} \leq \Aut(\Gamma)$ is the group of graph automorphisms and $H$ is the corresponding subgroup of $\Aut(W)$, then the natural action $W \circlearrowright X$ extends to a faithful geometric action $W \rtimes H  \circlearrowright X$.
\end{lemma}


\section{When $\phi$ is a partial conjugation}\label{sec:partialconjugation}


Now suppose that $\phi$ is the partial conjugation with acting letter $a$ and domain $D$.  Recall that $v_w$ denotes the vertex of $X$ indexed by the group element $w \in W$.  For any $d \in D$, $v_1$ and $v_d$ are adjacent in $X^1$, but $v_{\phi(1)}$ and $v_{\phi(d)}$ are not.  Since the map $v_w \mapsto v_{\phi(w)}$ does not respect adjacency in $X^1$, the left-muliplication action $W \circlearrowright X$ does not naturally extend to an action $G \circlearrowright X$.  However, $G$ is itself a right-angled Coxeter group, and hence also a CAT(0) group.

\begin{lemma}
\label{lem:extendbyPC}
If $\phi$ is a partial conjugation with acting letter $a$ and domain $D$, then $G$ is itself a right-angled Coxeter group.
\end{lemma}

We will omit the details of the proof, which may be found in \cite{OtherPreprint}.  In that paper, we engage more broadly with the problem of identifying a right-angled Coxeter presentation in a given group (or proving that no such presentation exists).  We find various families of extensions of right-angled Coxeter groups which are again right-angled Coxeter, and these include Lemma \ref{lem:extendbyPC} as a special case.

Here we will give a description of how to construct the defining graph $\Lambda$ for $G$ based on the original graph $\Gamma$.  The procedure is as follows:
\begin{enumerate}
\item Add a new vertex labeled $x$, which we connect to everything in $\Gamma \setminus D$.
\item Replace the label of vertex $a$ with the label $ax$, and add edges connecting $ax$ to each vertex in $D$.
\end{enumerate}
An example is shown in Figure \ref{fig:pcexample}.

\begin{figure}[h]
\begin{tikzpicture}[scale=0.7]

\node (a1) [circle,fill,inner sep=2pt,label=180:$a_1$] at (150:2) {};
\node (a2) [circle,fill,inner sep=2pt,label=90:$a_2$] at (110:1) {};
\node (a3) [circle,fill,inner sep=2pt,label=-90:$a_3$] at (0,0) {};
\node (a4) [circle,fill,inner sep=2pt,label=-90:$a_4$] at (2,0) {};
\node (a5) [circle,fill,inner sep=2pt,label=45:$a_5$] at (3,1) {};
\node (a6) [circle,fill,inner sep=2pt,label=0:$a_6$] at (4,0) {};

\draw (a1) to (a2) to (a3) to (a1);
\draw (a3) to (a4) to (a5);
\draw (a4) to (a6);

\node at (1,-1) {$\Gamma$};

\begin{scope}[xshift=9cm]
\node (a1) [circle,fill,inner sep=2pt,label=180:$a_1$] at (150:2) {};
\node (a2) [circle,fill,inner sep=2pt,label=90:$a_2$] at (110:1) {};
\node (a3) [circle,fill,inner sep=2pt,label=-90:$a_3$] at (0,0) {};
\node (a4) [circle,fill,inner sep=2pt,label=-90:$a_4$] at (2,0) {};
\node (a5) [circle,fill,inner sep=2pt,label=45:$a_5x$] at (3,1) {};
\node (a6) [circle,fill,inner sep=2pt,label=0:$a_6$] at (4,0) {};
\node (x) [circle,fill,inner sep=2pt,label=0:$x$] at (1,2.5) {};

\draw (a1) to (a2) to (a3) to (a1);
\draw (a3) to (a4) to (a5);
\draw (a4) to (a6);
\draw (x) to (a1);
\draw (x) to (a2);
\draw (x) to (a3);
\draw (x) to (a4);
\draw (x) to (a5);
\draw (a5) to (a6);

\node at (1,-1) {$\Lambda$};
\end{scope}

\end{tikzpicture}
\caption{$\Lambda$ is the defining graph of $W_{\Gamma} \rtimes \langle x \rangle$, where $x$ has acting letter $a_5$ and domain $\{a_6\}$.}
\label{fig:pcexample}
\end{figure}
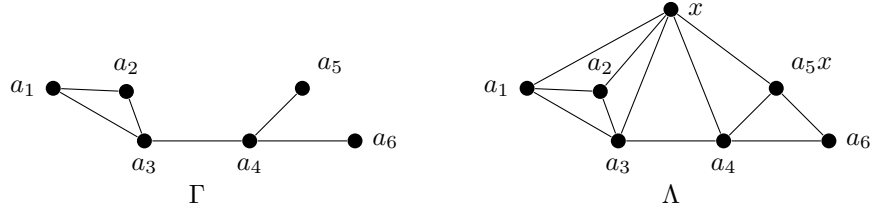

\section{When $\phi$ is a transvection}\label{sec:transvection}

Finally, we suppose that $\phi$ is the transvection with acting letter $a$ and domain $d$.
Recall that this means that $\St(d) \subseteq \St(a)$, and $\phi$ is determined by the rule:
\[
d \mapsto da, \text{ and } s \mapsto s \text{ for all } s \in S \setminus \{d\}.
\]

We note that $v_1$ and $v_d$ are adjacent in $X^1$, but $v_{\phi(1)}$ and $v_{\phi(d)}$ are not.  Since the map $v_w \mapsto v_{\phi(w)}$ does not respect adjacency in $X^1$, the left-multiplication action $W \circlearrowright X$ does not naturally extend to an action $G \circlearrowright X$. In fact, a stronger statement is true. It follows from \cite[Section 13.2]{Davis} that $\Fix(d)$ is a codimension 1 subspace of $\Sigma$, and $\Fix(da)$ is codimension 2.  Hence there is no isometry of $X$ which can conjugate the isometry representing $d$ to give the isometry representing $da$, so the left-multiplication action $W \circlearrowright X$ cannot be extended in any way to an action $G \circlearrowright X$.

We also note that $G$ does not embed in a right-angled Coxeter group since $G$ contains an element of order 4.  Since $xdx = ad$, we have that $(xd)^2 = a$ and $xd$ has order 4.  In a right-angled Coxeter group, any non-trivial element of finite order is an involution.

It seems that to show that $G$ is a CAT(0) group, we must identify a new CAT(0) space $Y$, and describe a faithful geometric action $G \circlearrowright Y$.  The key to our success in doing exactly this is the existence of a certain finite-index subgroup of $W$ which is itself a right-angled Coxeter group. Although the existence of such a subgroup is well-known (see \cite[Example 1.4]{BKS}, for example, where the analogous subgroup is used in the context of right-angled Artin groups), we provide the details here for completeness.

Let $h_a: W \to \mathbb{Z}/ 2 \mathbb{Z}$ denote the homomorphism determined by the rule: $a \mapsto 1$, and $s \mapsto 0$ for all $s \in S \setminus \{a\}$.  Let $U$ denote the kernel of $h_a$, and let
\[
S' = (S \setminus \{a\}) \cup \{asa \mid s \in S \setminus \St(a)\}.
\]

\begin{lemma}
The pair $(U, S')$ is a right-angled Coxeter system, and hence $U$ is a right-angled Coxeter group.  Further, conjugation by $a$ in $W$ restricts to an automorphism $\theta \in \Aut(U)$ induced by a permutation of $S'$; this automorphism is trivial if and only if $a$ is central in $W$.
\end{lemma}
\begin{proof}
If $a$ is central in $W$, then $S' = S \setminus \{a\}$, and the result is evident.  In this case conjugation by $a$ restricts to the trivial automorphism of $U$, and hence is the automorphism of $U$ induced by the trivial permutation of $S'$.

Suppose $a$ is not central in $W$.  An alternative presentation for $W$ may be constructed from the standard Coxeter presentation for $W$ by the following Tietze transformations:
\begin{itemize}
\item For each vertex $s \in S \setminus \St(a)$, introduce a new generator $\widehat{s}$, the defining relation $asa = \widehat{s}$, and redundant relations $a\widehat{s}a = s$ and $\widehat{s}^2 = 1$.
\item For each pair of adjacent vertices $s, t \in S \setminus \St(a)$, introduce the redundant relation $\widehat{s} \widehat{t} = \widehat{t} \widehat{s}$.
\item For each pair of adjacent vertices $s \in S\setminus \St(a)$ and $t \in \Lk(a)$, introduce the redundant relation $\widehat{s} t = t \widehat{s}$.
\item For each vertex $x \in \Lk(a)$, we rewrite the relation $xa=ax$ as $axa = x$.
\end{itemize}
The resulting presentation of $W$ is:
\begin{align*}
P_2 = \langle S' \cup\{a\} &\mid x^2 = 1 \text{ for all } x \in S', \\
&\quad [s,t] =1 \text{ for all } \{s, t\} \in E \textrm{ such that } s, t\neq a, \\
&\quad [\widehat{s}, \widehat{t}] = 1 \text{ for all } \{s, t\} \in E \textrm{ such that } s, t \in S\setminus \St(a),\\
&\quad [\widehat{s}, t] = 1 \text{ for all } \{s, t\} \in E \textrm{ such that } s \in S\setminus \St(a) \text{ and } t \in \Lk(a),\\
&\quad a^2 = 1, asa = s \text{ for all } s \in \Lk(a), \\
&\quad asa = \widehat{s} \text{ and } a \widehat{s} a = s \text{ for all } s \in S \setminus \St(a)\rangle.
\end{align*}
Evidently, this is the presentation of a semi-direct product in which the non-normal factor is $\langle a \rangle$, the normal factor is a right-angled Coxeter group with generating set
\[
S' = \left(S \setminus \{a\}\right) \cup \{\widehat{x} \mid x \in S \setminus \St(a)\},
\]
and $a$ acts on the normal factor as the automorphism $\theta$ induced by permuting the generators according to the rule
\[
x \mapsto \widehat{x} \text{ and } \widehat{x} \mapsto x \text{ for all } x \in S \setminus \St(a),  y \mapsto y \text{ for all } y \in \Lk(a).
\]
The action of $a$ on $U$ is non-trivial because $S \neq \St(a)$.
\end{proof}

We now have the following refined decomposition of $G$:
\[
G = \left(U \rtimes_\theta \langle a \rangle\right) \rtimes_\phi \langle z\rangle.
\]
A presentation $P_3$ for $G$ is obtained from the presentation $P_2$ for $W$ by appending the generator $z$ and relations
$$z^2 = 1, zsz = s \;\; \text{ for all } s \in S'\setminus \{d\}, zdz = da, zaz = a.$$

It follows that for each $g \in G$, there exist unique choices $u_g \in U$, and $\epsilon_g, \delta_g \in \{0, 1\}$, such that
$g = u_g a^{\epsilon_g} z^{\delta_g}$.  We shall write $Y$ for the CAT(0) cube complex on which $U$ acts geometrically and faithfully as defined in Section \ref{sec:RACG}, and we write $p: G \to U$ for the projection map $g \mapsto u_g$.  The projection map is not a homomorphism because for $s \in S \setminus \St(a)$ we have $p(a)p(s)p(a) = s \neq s' = p(s')$.  Even so, it allows us to parlay the left-multiplication action of $G$ on itself into an action of $G \circlearrowright Y$.

\begin{lemma}
For all $g \in S' \cup \{a, z\}$, the rule
\[
v_u \mapsto v_{p(gu)} \text { for all } u \in U,
\]
respects adjacency in $Y^1$, and hence determines an isometry $\Phi_g \in \Isom(Y)$.
\end{lemma}
\begin{proof}
Let $u \in U$, $s \in S'$ and $g \in S' \cup \{a, z\}$.  To prove the result we must establish that $v_{p(gu)}$ and $v_{p(gus)}$ are adjacent. For this it suffices to show that $(p(gu))^{-1}p(gus) \in S'$.

If $g \in S'$, then
\[
(p(gu))^{-1}p(gus) = (gu)^{-1}gus = s \in S'.
\]
If $g = a$, then
\begin{align*}
(p(au))^{-1}p(aus) &= (p(\theta(u)a))^{-1}p(\theta(us)a) \\
&= (\theta(u))^{-1} \theta(us) \\
&= \theta(s) \in S'.
\end{align*}

Finally, we consider the case $g = z$.  We note that if $d$ occurs an even number of times in any word for $u$, then $a$ occurs an even number of times in any word for $\phi(u)$, and $p(zu) = \phi(u)$. If, on the other hand, $d$ occurs an odd number of times in any word for $u$, then $a$ occurs an odd number of times in any word for $\phi(u)$, and $p(zu) = \phi(u)a$.  The parity of $d$ in a group element $u \in U$ is identified by the homomorphism $h_d: U \to \mathbb{Z} / 2\mathbb{Z}$ determined by the rule $d \mapsto 1$, and $s \mapsto 0$ for all $s \in S'\setminus\{d\}$.  Therefore, we consider cases based on the value of $h_d(u)$, and whether or not $s = d$.

If $h_d(u) = 0$ and $s \neq d$, then
$$(p(zu))^{-1} p(zus) = (\phi(u))^{-1} \phi(us) = s \in S'.$$


If $h_d(u) = 0$ and $s = d$, then
$$(p(zu))^{-1} p(zud) = (\phi(u))^{-1} \phi(ud)a = \phi(d)a = d \in S'.$$


If $h_d(u) = 1$ and $s \neq d$, then
$$(p(zu))^{-1} p(zus) = (\phi(u)a)^{-1} \phi(us)a = a \phi(u)^{-1} \phi(u) s a = asa = \theta(s) \in S'.$$


If $h_d(u) = 1$ and $s = d$, then
$$(p(zu))^{-1} p(zud) = (\phi(u)a)^{-1} \phi(ud) = a \phi(u)^{-1} \phi(u) da = ada = d \in S'.$$


Adjacency is respected in all cases, so the result holds in the case that $g = z$, and thus $\Phi_g$ is an isometry of $Y$ as required.
\end{proof}

In summary, we have that $G$ is presented by:
\begin{align*}
P_3 = \langle S' \cup\{a, z\} &\mid x^2 = 1 \text{ for all } x \in S', \\
&\quad [s,t] =1 \text{ for all } \{s, t\} \in E \textrm{ such that } s, t\neq a, \\
&\quad [\widehat{s}, \widehat{t}] = 1 \text{ for all } \{s, t\} \in E \textrm{ such that } s, t \in S\setminus \St(a),\\
&\quad [\widehat{s}, t] = 1 \text{ for all } \{s, t\} \in E \textrm{ such that } s \in S\setminus \St(a) \text{ and } t \in \Lk(a),\\
&\quad a^2 = 1, asa = s \text{ for all } s \in \Lk(a), \\
&\quad asa = \widehat{s} \text{ and } a \widehat{s} a = s \text{ for all } s \in S \setminus \St(a),\\
&\quad z^2 = 1, zsz = s \;\; \text{ for all } s \in S'\setminus \{d\}, zdz = da, zaz = a \rangle;
\end{align*}
and
\begin{align*}
&\Phi_s(v_u) = v_{su} \;\; \text{ for all } s \in S',\\
&\Phi_a(v_u) = v_{\theta(u)},\\
&\Phi_z(v_u) = v_{\phi(u)} \text{ if } h_d(u) = 0,\\
&\Phi_z(v_u) = v_{\phi(u)a} \text{ if } h_d(u) = 1.\\
\end{align*}

\begin{lemma}
The map
\[
g \mapsto \Phi_g \text{ for all } g \in S' \cup \{a, z\},
\]
determines a geometric action $G \circlearrowright Y$ which extends the left-multiplication action $U \circlearrowright Y$. If $a$ is not central in $W$, the action is faithful.  If $a$ is central in $W$, the kernel is the subgroup generated by $\{a, z\}$.
\end{lemma}
\begin{proof}
To prove that the map determines an isometric group action, we must prove that the relations in the presentation $P_3$ for $G$ hold when each $g \in S' \cup \{a, z\}$ is replaced by $\Phi_g$.  It is clear that those relations not involving either $a$ or $z$ remain true when each $g \in S'$ is replaced by $\Phi_g$.  We leave the reader to verify that the following relations hold (using the rules listed immediately before the statement of the lemma):
\begin{align*}
& \Phi_a^2 = 1,\\
& \Phi_a \Phi_s \Phi_a = \Phi_s \text{ for all } s \in \Lk(a),\\
& \Phi_a \Phi_s \Phi_a = \Phi_{\widehat{s}}  \text{ for all } s \in S \setminus \St(a),\\
& \Phi_a \Phi_{\widehat{s}} \Phi_a = \Phi_s \text{ for all } s \in S \setminus \St(a),\\
& \Phi_z^2 = 1, \\
& \Phi_z\Phi_s\Phi_z = \Phi_s \;\; \text{ for all } s \in S'\setminus \{d\},\\
& \Phi_z\Phi_d\Phi_z = \Phi_d\Phi_a,\\
& \Phi_z\Phi_a\Phi_z = \Phi_a.\\
\end{align*}

We note that, because $v_1 \mapsto  v_{p(g)}$, the stabilizer of $v_1$ is a subgroup of the finite abelian group $\langle a, z \rangle$.  If $a$ is not central in $W$, there exists $s \in S \setminus \St(a)$.  Computation shows that $\Phi_a, \Phi_{az}$ do not fix $v_{s}$, and $\Phi_z$ does not fix $v_{ds}$.  Our claims about the kernel of the action follow immediately.
\end{proof}




If $a$ is central in $W$, then there is no obvious way in which $a$ should act non-trivially on $Y$. We can, however, extend $Y$ to a new space $Y^+$ by appending two unit length edges in a ``v'' shape at each vertex, thereby providing pieces on which $a$ and $\phi$ can act non-trivially.  More formally, to construct $Y^+$ from $Y$ we write $v_u^0$ for $v_u$, and we append new vertices
\[
\left\{v_u^i \mid\text{for all } u \in U \text{ and } i \in \{-1, 1\} \right\},
\]
and new unit length edges
\[
\left\{\left\{v_u^0, v_u^{-1}\right\}, \left\{v_u^0, v_u^1\right\} \mid \text{for all } u \in U\right\}.
\]
It is evident that appending such ``v'' shapes at each vertex does not cause the CAT(0) property to fail, hence $Y^+$ is a CAT(0) cube complex.


\begin{proposition}
\label{Proposition:CasebIsCentral}
If $a$ is central in $W$, then $G$ acts faithfully and geometrically on $Y^+$.
\end{proposition}
\begin{proof}
Suppose that $a$ is central in $W$, i.e., that $\St(a) =\Gamma$.  Then $(U, S')$ is a right-angled Coxeter system, and $W = U \times \langle a\rangle$.

We now define a homomorphism $\Phi \colon G \to \Isom(Y^+)$. For each $s \in S'$, we declare $\Phi(s)$ to be the isometry determined by the rule:
\[
v_u^i \mapsto v_{su}^i \text{ for all }  u \in U \text{ and } i \in \{-1, 0, 1\}.
\]
We declare $\Phi(a)$ to be the isometry determined by the rule:
\[
v_u^{i} \mapsto v_u^{-i} \text{ for all } u \in U \text{ and } i \in \{-1, 0, 1\}.
\]
We declare $\Phi(z)$ to be the isometry determined by the rule:
\[
v_u^i \mapsto \begin{cases}
v_u^i & \text{if } h_d(u) = 0,\\
v_u^{-i} & \text{if } h_d(u) = 1, \\
\end{cases}
\]
for all $u \in U \text{ and } i \in \{-1, 0, 1\}.$  The maps can be described informally as follows: each $s \in S'$ acts on $Y^+$ in the way which most naturally extends the left-multiplication action $U \circlearrowright Y$; $a$ flips the ``v'' attached to every vertex; while $z$ flips only half the ``v'' shapes, because it flips the ``v'' attached to a vertex $v_u$ if and only if $d$ has an odd parity in $u$.

It is evident that the maps described above preserve adjacency in the one-skeleton of $Y^+$, and hence determine isometries of $Y^+$.  Simple computations confirm that these definitions respect the relations in the presentation $P_3$ of $G$ (some of the relations listed are vacuous).  Therefore these definitions do indeed determine an isometric action $G \circlearrowright Y^+$.   That the action is geometric follows easily from the fact that the left-multiplication action $U \circlearrowright Y$ is geometric.
\end{proof}

\bibliographystyle{plain}
\bibliography{CATZeroExtensionsOfRACGBib}
\end{document}